\documentclass[EJP,preprint]{ejpecp}

\usepackage[T1]{fontenc}
\usepackage[utf8]{inputenc}
\usepackage{color}
\usepackage[square,numbers]{natbib}

\newtheorem{thm}{Theorem}[section]
\newtheorem{lem}[thm]{Lemma}
\newtheorem{prop}[thm]{Proposition}
\newtheorem{cor}[thm]{Corollary}
\theoremstyle{definition}

\newtheorem{rem}{Remark}[section]

\newtheorem{hyp}{Hypothesis}[section]

\numberwithin{equation}{section}

\newcommand\disp{\displaystyle}

\newcommand{\dB}{\ensuremath{\mathbb{B}}}

\newcommand{\dF}{\ensuremath{\mathbb{F}}}

\newcommand{\I}{\ensuremath{\mathbb{I}}}

\newcommand{\dR}{\ensuremath{\mathbb{R}}}

\newcommand{\dZ}{\ensuremath{\mathbb{Z}}}

\newcommand{\cA}{\ensuremath{\mathcal{A}}}

\newcommand{\cC}{\ensuremath{\mathcal{C}}}

\newcommand{\cF}{\ensuremath{\mathcal{F}}}

\newcommand{\cJ}{\ensuremath{\mathcal{J}}}

\newcommand{\cU}{\ensuremath{\mathcal{U}}}
\newcommand{\cV}{\ensuremath{\mathcal{V}}}

\newcommand{\cZ}{\ensuremath{\mathcal{Z}}}

\newcommand\nt{{\lfloor n t \rfloor}}
\newcommand\sn{n^{\frac{1}{2}}\; }
\newcommand\sni{n^{-\frac{1}{2}}\; }

\newcommand\setd{\{1,\ldots,d\}}

\newcommand{\BB}[1]{\textcolor{black}{#1}}

\usepackage{enumerate}

\SHORTTITLE{CLT for martingales-I : continuous limits}

\TITLE{Central limit theorems for martingales-I : \\
continuous limits\
    \thanks{Partial funding in support of this work was provided by the Natural
Sciences and Engineering Research Council of Canada.}
}
\AUTHORS{
Bruno R\'{e}millard\footnote{HEC Montr\'{e}al, Canada. \EMAIL{bruno.remillard@hec.ca}}
\and
Jean Vaillancourt\footnote{HEC Montr\'{e}al, Canada. \EMAIL{jean.vaillancourt@hec.ca}}
}
\KEYWORDS{Brownian motion; stochastic processes; weak convergence; martingales; mixtures}
\AMSSUBJ{60G44}
\AMSSUBJSECONDARY{60F17}
\SUBMITTED{August 12, 2023}
\ARXIVID{2301.07267}

\ABSTRACT{When the limiting compensator of a sequence of martingales is continuous, we obtain a weak convergence theorem for the martingales; the limiting process can be written as a Brownian motion evaluated at the compensator and we find sufficient conditions for  both processes to be  independent. As examples of applications, we revisit some known results for the occupation times of Brownian motion and symmetric random walks. In the latter case, our proof is much simpler than the construction of strong approximations. Furthermore, we  extend finite dimensional convergence of statistical estimators of financial volatility measures to convergence as stochastic processes.}

\begin{document}


\section{Introduction}

Many central limit theorems involving a sequence of martingales are about weak convergence to a Brownian motion, and often they are obtained by requesting that the associated sequence of compensators converge in probability to a deterministic function. There are also some results where the sequence of compensators converge in probability to a stochastic process, and then the limit can be written as a mixture of a Brownian motion independent of the limiting compensator. Such results appear naturally for occupation times of symmetric random walks \citep{Csaki/Foldes:1998, Csaki/Foldes:2000} through complicated constructions of related strong approximations.
Our goal here is to gather, under the umbrella of a single martingale central limit theorem (CLT), a number of weak convergence results
for sequences of real-valued stochastic processes to mixtures of Brownian motions and independent increasing processes, without relying to special properties of the processes, nor involving complicated constructions. In fact, under  weak conditions involving sequence of martingale compensators,
we will  show that the CLTs  holds for general martingales, and that the limiting process is a mixture of a Brownian motion with the limiting compensator of the sequence of martingales, and both processes are independent. These conditions are easy to verify and are general enough to be applicable to a wide range of situations.

The main results are stated in Section \ref{sec:main}, while some examples involving either independent and identically distributed (iid) sequences or Markov processes,
are found in Section \ref{sec:xmpls}. A longer application regarding volatility modeling in mathematical finance
is in Section \ref{sec:volatility_stats}, extending previous results of \cite{Barndorff-Nielsen/Shephard:2002, Barndorff-Nielsen/Shephard:2003a}.

\section{Main results}\label{sec:main}
Let $D= D[0,\infty)$ be the Polish space of $\dR^d$-valued c\`adl\`ag trajectories (right continuous with left limits everywhere),
equipped with the Skorohod's $\cJ_1$-topology on $D$ ---
consult either \citet{Ethier/Kurtz:1986} or \citet{Jacod/Shiryaev:2003} for additional insight, as well as any unexplained terminology.
All processes considered here have their trajectories in $D$ and are adapted to a filtration $\dF = (\cF_t)_{t\ge 0}$ on a probability space
$(\Omega,\cF,P)$ satisfying the usual conditions (notably, right continuity of the filtration and completeness).
Trajectories in $D$ are usually noted $x(t)$ but occasionally $x_t$.
Weak convergence of a sequence of $D$-valued processes $X_n$ to another such process $X$, this last with continuous trajectories,
will be denoted by $X_n \stackrel{\mathcal{C}}{\rightsquigarrow} X$, while the weaker convergence of finite dimensional distributions will be denoted by
$X_n \stackrel{f.d.d}{\rightsquigarrow} X$. In this paper the limit $X$ always has continuous trajectories
(even though the sequence of processes $X_n$ may well not) so weak convergence in the $\cJ_1$-topology coincides with that in the
$\mathcal{C}$-topology, induced by the supremum norm over compact time sets. We will refer to $\mathcal{C}$-tightness etc. without further ado.
All processes are written in coordinatewise fashion such as $X_n=(X_n^i)_{1\le i\le d}$.
Writing $|\cdot|$ for the Euclidean norm, square integrable processes $X$ are those satisfying $E\{|X(t)|^2\}<\infty$ for every $t\ge0$, while
$L_2$-bounded ones also satisfy $\sup_{t\ge0}E\{|X(t)|^2\}<\infty$.

Suppose that $M_n$ is a sequence of $D$-valued square integrable $\mathbb{F}$-martingales started at $M_n(0)=0$.
Because of the discontinuity of trajectories, the (matrix-valued or cross) quadratic variation
$[M_n]:=([M_n^i,M_n^j])_{1\le i, j\le d}$ is distinct from its matrix-valued (predictable) compensator
$A_n:=\langle M_n \rangle $ --- another writing economy which means
$(A_n^{ij})_{1\le i, j\le d}:=(\langle M_n^i,M_n^j  \rangle)_{1\le i, j\le d}$ ---
and it is the latter that is of interest from a practical point of view.
Coordinatewise, we use $[M_n^i]=[M_n^i,M_n^i]$ and $\langle M_n^i \rangle = \langle M_n^i,M_n^i \rangle $ as well.
The largest jump of $X\in D$ over $[0,t]$ is denoted by $J_t(X) := \sup_{s\in [0,t]}|X(s)-X(s-)|$.
The following assumption will be used repeatedly.

\begin{hyp}\label{hyp:an_unbounded}
All of the following hold:
\begin{enumerate}
\item[(a)] $A_n^{ii}(t) = \langle M_n^i \rangle_t \to\infty$ as $t\to\infty$ almost surely, for each fixed $n\ge1$ and $i\in\{1,2,\ldots,d\}$;
\item[(b)] There is a $D$-valued process $A$ such that \\
(i) $A_n \stackrel{f.d.d.}{\rightsquigarrow} A$ and $A^{ij}=0$ for all $i\neq j$; \\
(ii) for all $t\ge 0$ and $i$, $\lim_n E\left\{A^{ii}_n(t)\right\}=E\left\{A^{ii}(t)\right\} < \infty$; \\
(iii) for all $i$, $A^{ii}(t) \to\infty$ as $t\to\infty$ almost surely.
\end{enumerate}
\end{hyp}
Writing the inverse process for $A_n^{ii}$ as $\tau_n^i(s) = \inf\{t\ge0; A_n^{ii}(t)>s\}$,
one defines the rescaled $\mathbb{F}_{\tau_n^i}$-martingale $W_n^i = M_n^i\circ \tau_n^i$, with compensator $A_n^{ii}\circ \tau_n^i$.
Note that by definition and using the right-continuity of $A_n^{ii}$, $A_n^{ii}\circ \tau_n^i(t)\ge t$.
Actually, $W_n^i$ is an $\mathbb{F}_{\tau_n^i}$-Brownian motion
with respect to the filtration $\mathbb{F}_{\tau_n^i}=\{\cF_{\tau_n^i(t)}\}_{t\ge 0}$,
whenever Hypothesis \ref{hyp:an_unbounded}.a holds and $M_n^i$ is continuous everywhere,
by \citet{Dambis:1965} or \citet{Dubins/Schwarz:1965}. In the latter case, the continuity of $M_n^i$
implies the continuity of both $[M_n^i]$ and $\langle M_n^i \rangle$,
as well as their equality $[M_n^i]=\langle M_n^i \rangle$. In this special case, the sequence $M_n$ therefore
comprises a naturally associated sequence of Brownian motions $W_n^i$ coordinatewise.
This is no longer the case as soon as at least one of the $M_n^i$'s has a discontinuity anywhere.
Obtaining a CLT therefore requires building such a Brownian motion,
possibly on an enlargement of the stochastic basis $(\Omega,\cF,\dF, P)$ with $\dF = (\cF_t)_{t\ge 0}$.
Such enlargements will be used systematically in this paper and are understood to affect some statements implicitly,
without further ado, for instance in some of the proofs.
Equality in law is denoted by $\stackrel{Law}{=}$;
convergence in probability is denoted by $\stackrel{Pr}{\to}$, in law by $\stackrel{Law}{\to}$, and
almost sure convergence by $\stackrel{a.s.}{\to}$.

\begin{thm}\label{thm:clt1}
Assume that Hypothesis \ref{hyp:an_unbounded} holds with $A$ continuous everywhere; that
$J_t(M_n) \stackrel{Law}{\to} 0$ for any $t>0$; that there exists an
$\dF$-adapted sequence of $D$-valued square integrable martingales $B_n$ started at $B_n(0)=0$ so that
\begin{enumerate}
\item $(A_n, B_n) \stackrel{\mathcal{C}}{\rightsquigarrow} (A,B)$ holds, where $B$ is a Brownian motion
with respect to its natural filtration $\dF_B = \{\cF_{B,t}: \; t\ge 0\}$ and $A$ is $\dF_B$-measurable;
\item  $\langle M_n^i, B_n^j \rangle_t \stackrel{Law}{\to} 0$, for any $i,j\in\setd$ and $t\ge 0$.
\end{enumerate}
Then $(M_n, A_n,B_n) \stackrel{\mathcal{C}}{\rightsquigarrow} (M,A,B)$ holds,
where $M$ is a continuous square integrable ${\cF}_{t}$-mar\-tin\-ga\-le with respect to (enlarged) filtration
with predictable quadratic variation process $A$.
Moreover, $M^i = W^i\circ A^{ii}$, $i\in \setd$ holds,
with $W$ a standard Brownian motion which is independent of $B$ and $A$.
\end{thm}

No need for $A_n$ to converge in probability, nor for the nested filtrations required for stable convergence
\citep[Section VIII.5c]{Jacod/Shiryaev:2003}, the usual way to characterize the law of $M$ uniquely.
\begin{proof}
\BB{The idea of the proof is to use the convergence in law of $A_n$ to a continuous $A$ to get the tightness of $M_n$. Then, on another probability space, there is a subsequence converging almost surely. It remains to show that any possible limit of any subsequence has the desired properties.}
First, by Proposition \ref{prop:Cincreasing}, the continuity of $A$ implies that
$A_n^{ii} \stackrel{\mathcal{C}}{\rightsquigarrow} A^{ii}$ holds for every $i$.
By \citet[Theorem I.4.2]{Jacod/Shiryaev:2003}, one concludes
$(A_n^{ij}(t)-A_n^{ij}(s))^2 \le (A_n^{ii}(t)-A_n^{ii}(s)) (A_n^{jj}(t)-A_n^{jj}(s))$ almost surely, for every choice of $i$ and $j$,
so each $A_n^{ij}$ is $\mathcal{C}$-tight and hence $A_n \stackrel{\mathcal{C}}{\rightsquigarrow} A$ as well.
Next, set $X_n^{i}(t) = \left[M_n^{i}\right]_{\tau_n+t}-\left[M_n^{i}\right]_{\tau_n}$ and
$Y_n^{i}(t) = A_n^{ii}(\tau_n+t)-A_n^{ii}(\tau_n)$,
where $\tau_n$ is a sequence of bounded stopping times. The filtration of interest here is $\{\cF_{\tau_n+t}\}_{t\ge 0}$.
For any $\epsilon>0$, $\eta>0$ and  sequence $\delta_n \in (0,1) \to 0$,
$E X_n^{i}(\delta_n) = E Y_n^{i}(\delta_n)$ implies
 \begin{equation}\label{[Mn]J1tight}
 P\left(\left[M_n^{i}\right]_{\tau_n+\delta_n}-\left[M_n^{i}\right]_{\tau_n} \geq\varepsilon\right) \leq \frac{1}{\varepsilon}
\left[\eta+ E\left\{J_{\delta_n}(Y_n^{i}) \right\}\right] + P(Y_n^{i}(\delta_n))\geq\eta),
\end{equation}
using \eqref{eng2} of Lemma \ref{lem:lenglart}, with $J_{\delta_n}(Y_n^{i})\le Y_n^{i}(\delta_n)\le\omega_\cC(A_n^{ii},\delta_n,T)$,
since $A_n^{ii}$ is nondecreasing. Therefore, the continuity of the limit $A^{ii}$ implies $Y_n^{i}(\delta_n)\stackrel{Law}{\to}0$
while $E\left\{Y_n^{i}(\delta_n)\right\}\to0$ follows from Hypothesis \ref{hyp:an_unbounded}.b(ii) by the dominated convergence theorem
\citep[Proposition App1.2]{Ethier/Kurtz:1986}. By Theorem \ref{thm:M1tightness}, $\left[M_n^{i}\right]$ is $\cJ_1$-tight.
Since $J_t(\left[M_n^{i}\right]) \stackrel{Law}{\to} 0$ is assumed, $\left[M_n^{i}\right]$ is actually $\mathcal{C}$-tight and hence
so are $M_n^{i}$ and $(M_n^{i},A_n^{ii},\left[M_n^{i}\right])$ successively. For every choice of $i$ and $j$, inequality
$([M_n^{i},M_n^{j}]_t-[M_n^{i},M_n^{j}]_s)^2 \le (\left[M_n^{i}\right]_t-\left[M_n^{i}\right]_s) ([M_n^{j}]_t-[M_n^{j}]_s)$
ensures the individual $\mathcal{C}$-tightness of each $[M_n^{i},M_n^{j}]$.
The $\mathcal{C}$-tightness of matrix $[M_n]$ follows, implying that of $M_n$ and $(M_n,A_n,[M_n])$.
By Skorohod's theorem, e.g., \citet[Theorem 3.1.8]{Ethier/Kurtz:1986}, there exists a subsequence
$\{n_k\}$, a probability space $(\Omega^\prime,{\cF}^\prime,P^\prime)$, and $D$-valued processes
$Z^\prime_{n_k}:=(M^\prime_{n_k},B^\prime_{n_k},A^\prime_{n_k})$ and $Z^\prime:=(M^\prime,B^\prime,A^\prime)$
defined on $(\Omega^\prime,{\cF}^\prime,P^\prime)$ which are such that $Z^\prime_{n_k}$ and
$Z_{n_k}:=(M_{n_k},B_{n_k},\langle M_{n_k} \rangle )$ are identical in law for all $k \ge 1$, and
$Z^\prime_{n_k}$ converges almost surely to $Z^\prime$ uniformly on compact time sets.
We next prove that, for any such limit point $Z^\prime=(M^\prime,B^\prime,A^\prime)$,
$M^\prime$ and $B^\prime$ are both martingales with respect to the natural filtration
$\cF_t^\prime= \sigma\{M^\prime(s),B^\prime(s),A^\prime(s); s\le t\}$,
such that $\langle M^\prime \rangle_t  =A^\prime_t$, $\langle B^\prime \rangle_t=t$,  and $\langle M^\prime,B^\prime \rangle _t\equiv 0$.

\BB{For a given $\ell\ge1$ and $0\le s_1 \le \cdots \le s_{\ell} \le s\le t$, set $Y_{n_k} = \left( Z_{n_k}(s_1), \ldots, Z_{n_k}(s_\ell)\right)$,
$Y^\prime_{n_k} = \left( Z^\prime_{n_k} (s_1), \ldots, Z^\prime_{n_k}(s_\ell)\right)$ and $Y^\prime = \left( Z^\prime (s_1), \ldots, Z^\prime(s_\ell)\right)$. Then, for any continuous and bounded function $f$ of $\ell\times\{(1+d)^2-1\}$ variables,  $\left(M^\prime_{n_k}(t)-M^\prime_{n_k}(s)\right)f\left(Y^\prime_{n_k}\right)$ converges almost surely to
$\left(M^\prime(t)-M^\prime(s)\right)f\left(Y^\prime\right)$  and it is uniformly integrable since
$E\left\{\left( M^\prime_{n_k}(t) \right)^2\right\} = EA^\prime_{n_k}(t) = EA_{n_k}(t) \to EA(t)$.  }
\BB{For $M^\prime$ to be a $\cF^\prime$-martingale, it suffices to prove $E\left\{ \left(M^\prime(t)-M^\prime(s)\right)f(Y^\prime) \right\}=0$ for any bounded (by $1$) and continuous $f$, because the cylinder sets $\{Y^\prime\in O\}$, $O$ open, generate the $\sigma$-algebra
$\cF^\prime(s)$, and because  $\I_O$ can be approximated by a sequence of bounded and continuous functions.}
Without loss of generality, we do so coordinatewise (hence $d=1$) and drop the superscript $i$ for the rest of the proof,
in order to keep notation to a minimum. 
\BB{
Since $A_{n_k}(u)$ is $\cF_{n_k}(s)$-measurable for all $ 0\le u \le s$, $f(Y_{n_k})$ is $\cF_{n_k}(s)$-measurable and
\begin{eqnarray*}
0 = E\left\{ \left(M_{n_k}(t)-M_{n_k}(s)\right)f(Y_{n_k}) \right\}
&=& E\left\{ \left(M^\prime_{n_k}(t)-M^\prime_{n_k}(s)\right)f\left(Y^\prime_{n_k}\right) \right\} \\
& & \stackrel{k\to\infty}{\rightarrow} E\left\{ \left(M^\prime(t)-M^\prime(s)\right)f\left(Y^\prime\right) \right\}.
\end{eqnarray*}
}
Moreover, using \citet[Proposition App2.3]{Ethier/Kurtz:1986} and Hypothesis \ref{hyp:an_unbounded}.b,
$\{M^\prime_{n_k} (t)\}^2$ \BB{is uniformly integrable } for each $t\ge 0$. Next, for all $ 0 \le s\le t$,
\begin{eqnarray}
E\left\{\left\{M^\prime (t)-M^\prime(s)\right\}^2\right\} &=& \lim_{k\to \infty}
E\left\{\left\{M^\prime_{n_k} (t) - M^\prime_{n_k} (s)\right\}^2\right\} \nonumber \\
&=& \lim_{k\to \infty}  E\left\{A^\prime_{n_k} (t) - A^\prime_{n_k} (s)\right\}  =  E\left\{A(t)-A(s)\right\}.\label{eq:conv_quad}
\end{eqnarray}
Therefore $M^\prime$ is a square integrable martingale with respect to filtration
$\{\cF^\prime(t)\}_{t\ge 0}$. The same argument also applies to $B^\prime$.
\BB{Next,
$
\left\{ \left(M^\prime_{n_k}(t)-M^\prime_{n_k}(s)\right)^2-A^\prime_{n_k}(t) + A^\prime_{n_k}(s)\right\}f(Y^\prime_{n_k})$
converges almost surely to
$\left\{\left(M^\prime(t)-M^\prime(s)\right)^2-A^\prime(t)+A^\prime(s)\right\}f(Y^\prime)$} and its absolute value
 is bounded by
$\left(M^\prime_{n_k}(t)-M^\prime_{n_k}(s)\right)^2 +A^\prime_{n_k}(t)$,
which converges almost surely to
$\left(M^\prime(t)-M^\prime(s)\right)^2+A^\prime(t)$. Using \eqref{eq:conv_quad}, as $k\to\infty$, we have
$$
E\left\{\left(M^\prime_{n_k}(t)-M^\prime_{n_k}(s)\right)^2 +A^\prime_{n_k}(t)\right\} = 2EA^\prime_{n_k}(t)-EA^\prime_{n_k}(s)
\to E\left\{\left(M^\prime(t)-M^\prime(s)\right)^2 +A^\prime(t)\right\}.
$$
By dominated convergence \citet[Proposition App1.2]{Ethier/Kurtz:1986}, one can conclude
that
\BB{\begin{eqnarray*}
0 &=& E\left\{ \left\{ \left(M^\prime_{n_k}(t)-M^\prime_{n_k}(s)\right)^2-A^\prime_{n_k}(t) +
A^\prime_{n_k}(s)\right\}f(Y^\prime_{n_k}) \right\} \\
&\to &
E\left\{\left\{\left(M^\prime(t)-M^\prime(s)\right)^2-A^\prime(t)+A^\prime(s)\right\}f(Y^\prime)\right\}.
\end{eqnarray*}}
It follows that $A^\prime$ is the quadratic variation process of the martingale $M^\prime$ since, by construction,
$A^\prime(t)$ is $\cF^\prime(t)$-measurable.  Again, the same argument
holds true for $B^\prime$, with quadratic variation $\langle B\rangle_t=t$, $t\ge 0$.
Finally, both the within and cross off-diagonal terms in $\langle \tilde M_{n_k}, B_{n_k} \rangle $, are taken care of coordinatewise,
as in the preceding argumentation, setting $d=1$.
For instance, the copy of the cross term
\BB{$$
 \left\{  M^\prime_{n_k}(t)B^\prime_{n_k}(t)-M^\prime_{n_k}(s)B^\prime_{n_k}(s)
 - \langle \tilde M_{n_k}^\prime, B^\prime_{n_k} \rangle _t
 + \langle \tilde M_{n_k}^\prime, B^\prime_{n_k} \rangle _s \right\}f(Y^\prime_{n_k})
 $$
converges almost surely to $\left\{  M^\prime(t)B^\prime(t)-M^\prime(s)B^\prime(s)\right\}f(Y^\prime)$,}
and its absolute value, using Kunita-Watanabe's inequality,  is bounded by
$$
g_{n_k}= \frac{1}{2} \left(M^\prime_{n_k}(t)\right)^2 + \frac{1}{2} \left(M^\prime_{n_k}(s)\right)^2
+ \frac{1}{2} \left(B^\prime_{n_k}(t)\right)^2 + \frac{1}{2} \left(B^\prime_{n_k}(s)\right)^2
+ \frac{1}{2} A^\prime_{n_k}(t)  + \frac{1}{2} \langle B'_{n_k} \rangle_t,
$$
which converges almost surely to
$$
g= \frac{1}{2} \left(M^\prime(t)\right)^2 + \frac{1}{2} \left(M^\prime(s)\right)^2
+ \frac{1}{2} \left(B^\prime(t)\right)^2 + \frac{1}{2} \left(B^\prime(s)\right)^2 + \frac{1}{2} t+ \frac{1}{2} A^\prime(t).
$$
Hypothesis (b) implies $E(g_{n_k})\to E(g)$, so
\BB{$$
 E\left\{ \left\{  M^\prime_{n_k}(t)B^\prime_{n_k}(t)-M^\prime_{n_k}(s)B^\prime_{n_k}(s)
 - \langle \tilde M_{n_k}^\prime,B^\prime_{n_k} \rangle_t
 + \langle \tilde M_{n_k}^\prime, B^\prime_{n_k} \rangle_s \right\}f(Y^\prime_{n_k})\right\} \equiv 0
$$
converges to
$E\left\{ \left\{  M^\prime(t)B^\prime(t)-M^\prime(s)B^\prime(s)\right\}f(Y^\prime)\right\}$} and hence
$\langle M^\prime,B^\prime \rangle =0$.
Thus any limit point $Z^\prime = (M^\prime,B^\prime,A^\prime)$ has the property that
$M^\prime$, $B^\prime$, and $M^\prime B^\prime$  are martingales with respect to the natural filtration $\cF_t^\prime$,
with $\langle M^\prime \rangle  = A^\prime$, $B^\prime$ is a Brownian motion, and most importantly,
$\langle M^\prime,B^\prime \rangle_t\equiv 0$. Since we already know that the trajectories of $M^\prime$ and $B^\prime$ are continuous,
it follows from \citet[Theorem II.7.3]{Ikeda/Watanabe:1989} that $M^\prime  = W^\prime\circ A^\prime$
with independent Brownian motions $W^\prime$ and $B^\prime$ on probability space $(\Omega^\prime,{\cF}^\prime,P^\prime)$.
Since, $A^\prime$ is $\dF_{B^\prime}$-measurable by hypothesis, $W^\prime$ is also independent of $A^\prime$.
Therefore all limit points $(M^\prime,B^\prime,A^\prime)$ have the same law
since the law of $A^\prime$ is the same has the one of $A$.
\end{proof}

\begin{rem}
An historically important prototype of Theorem \ref{thm:clt1} is Rebolledo's landmark CLT for local martingales, when restricted to sequences of square integrable martingales \citep{Rebolledo:1980} satisfying an asymptotic rarefaction of jumps condition.
Functional CLTs involving limiting mixtures with non deterministic $A$ go back to \citet[p. 92-93]{Rebolledo:1979} for processes converging to a diffusion in $\dR^d$
 and \citet{Johansson:1994}, for point process martingales.
Additional references on the early successes in the discrete case can be found in \citet{Hall/Heyde:1980}
and in the continuous case in \citet[Section VIII.5]{Jacod/Shiryaev:2003}. More recently,
\citet{Merlevede/Peligrad/Utev:2019} display the current state of the art of the functional CLT
for rescaled non-stationary sequences and arrays of dependent random variables when the limit is continuous
--- either a Brownian motion or the stochastic integral of a continuous function with respect to a Brownian motion.
 \end{rem}

The proof of the following useful proposition 
relies on the tightness induced by the convergence of the quadratic variation process.
\begin{prop}\label{prop:stochInt}
Suppose $(\xi_j)_{j\ge 1}$ is a sequence of iid random variables with mean $0$ and variance $1$, independent of a
continuous stochastic process $\sigma$ defined on $[0,\infty)$. Set $\disp M_n(t) = \sni\sum_{j=1}^\nt \sigma\left(\frac{j-1}{n}\right)\xi_j$,
$\disp V_n(t) = n^{-1}\sum_{j=1}^\nt \sigma^2\left(\frac{j-1}{n}\right)$, and define $\disp V(t) = \int_0^t \sigma^2(s)ds$. Finally, set
$\disp B_n(t) = \sni\sum_{j=1}^\nt \xi_j$.
Then \BB{$(M_n,V_n,B_n) \stackrel{\cC}{\longrightarrow} (M,V,B)$,
where $B$ is a Brownian motion independent of $\sigma$ and $M$ can also be written as a stochastic integral
with respect to $B$ viz. $M(t) = \int_0^t \sigma(s)dB_s$, $M = W\circ V$, and $W$ is a Brownian motion independent of $V$. }
\end{prop}
\begin{proof}
Let $\cF_{n,t} = \sigma\left\{\xi_j, a\left(\frac{j}{n}\right); j\le \nt\right\}$.
Then $M_n$ is a $\cF_{n,t}$-martingale with
$\disp \langle M_n\rangle(t) =n^{-1}\sum_{j=1}^\nt \sigma^2\left(\frac{j-1}{n}\right)=V_n(t)$.
From the continuity of $\sigma$, one gets that $V_n\to V$ and $V$ is continuous.
Also, if $\disp B_n(t) = \sni \sum_{j=1}^\nt \xi_j$, then $(B_n,V_n) \stackrel{\cC}{\longrightarrow}
(B,V)$, where $B$ is a Brownian motion independent of $\sigma$ and $V$.
Since $J_t(M_n) \stackrel{Law}{\to} 0$ holds, Theorem \ref{thm:M1tightness} shows that $(M_n,V_n, B_n)$ is $\cC$-tight.
To complete the proof, let $W$ be a Brownian motion independent of $V$. It is sufficient to show that
for any  $0=t_0<t_1<\cdots < t_m$,
$M_n(t_1),\ldots, M_n(t_m)$, $V_n(t_1),\ldots, V_n(t_m)$, and $B_n(t_1),\ldots, B_n(t_m)$ converge jointly in law to $W\circ V(t_1),\ldots, B(t_m)$. To this end,
take $\theta_1, \eta_1, \lambda_1,\ldots, \theta_m, \eta_m,\lambda_m \in \dR$, and
set $\varphi(s) = E\left(e^{is \xi_j}\right)$. Next, setting $G_n = e^{i  \sum_{j=1}^m  \eta_{j}\{V_n(t_j)-V_n(t_{j-1})\} }$, and using the standard proof of the CLT, one gets
\begin{multline*}
E\left[ e^{ i \sum_{k=1}^m \left[  \lambda_{k}\{M_n(t_k)-M_n(t_{k-1})\}+  \theta_{k}\{B_n(t_k)-B_n(t_{k-1})\} +  \eta_{k}\{V_n(t_k)-V_n(t_{k-1})\}\right]} \right] \\
= E\left[G_n \prod_{k=1}^m \prod_{j= \lfloor nt_{k-1}\rfloor +1}^{\lfloor nt_{k}\rfloor}  \varphi\left\{\sni\theta_k + \sni \lambda_{k}\sigma\left(\frac{j-1}{n} \right)\right\} \right]\\
=  E\left[G_n e^{-\frac{1}{2} \sum_{k=1}^m  \lambda_k^2 \left\{V_n(t_k)-V_n(t_{k-1})\right\} -\frac{1}{2} \sum_{k=1}^m  \theta_k^2(t_k-t_{k-1})}\right]+o(1)\\
\to E\left[e^{-\frac{1}{2}\sum_{k=1}^m  \lambda_k^2\{V(t_k)-V(t_{k-1})\} -\frac{1}{2} \sum_{k=1}^m  \theta_k^2(t_k-t_{k-1}) +i \sum_{j=1}^m  \eta_{k}\{V(t_k)-V(t_{k-1})\} }\right]\\
= E\left[e^{i \sum_{k=1}^m \left[ \lambda_{k}\int_{t_{k-1}}^{t_k} \sigma(s)dB_s +  \eta_{k}\{V(t_k)-V(t_{k-1})\} +  \eta_{k}\{B(t_k)-B(t_{k-1})\}\right] }\right].
\end{multline*}
Taking $\theta_1=\cdots=\theta_m=0$, one gets that $M$ has also the same distribution as $ W\circ V$, for a Brownian motion $W$ independent of $V$.
\end{proof}

\section{Examples of application to occupation times}\label{sec:xmpls}

\subsection{Occupation times for Brownian motion}\label{ssec:BM}
Let $B$ denote a Brownian motion, $V$ continuous with compact support, and $\mu_V = \int_{-\infty}^\infty V(y)dy =0$ but
$V$ is not identically $0$. Set $F(x) = -2\int_{-\infty}^x V(y)dy$ and $G(x) =  \int_{-\infty}^x F(y)dy$ and consider the martingale
$M_1(t) = G(B_t) +\int_0^t V(B_s)ds$ and the occupation time $\int_0^t V(B_s)ds$. Setting $B_n(t) = n^{-\frac{1}{2}}B_{nt}$, then
the continuous martingale $M_n(t) =n^{-\frac{1}{4}} \int_0^{nt} F(B_s)dB_s = n^{\frac{1}{4}} \int_0^{t} F\left(n^{\frac{1}{2}}B_n(s)\right)dB_n(s)$ has the same asymptotic behavior as
$n^{-\frac{1}{4}} \int_0^{nt}V(B_s)ds$ since $G$ is bounded. Using the scaling property of Brownian motion, it follows that
$\left( M_n, [M_n ], B_n \right)_{\BB{t\ge 0}}   \stackrel{Law }{=}  \left(\tilde  M_n,[\tilde M_n ], \tilde B\right)_{\BB{t\ge 0}}$, where
$\tilde  B$ is another Brownian motion, $\disp \tilde  M_n(t) =  n^{\frac{1}{4}} \int_0^t F(\sqrt n \tilde B_u)d\tilde B_u$, and $\disp [\tilde M_n ]_t =  n^{\frac{1}{2}} \int_0^tF^2(\sqrt n \tilde B_u)du$.
Hence, as $n\to\infty$,
$\disp
A_n (t)=  [\tilde M_n ]_t  =  \sqrt n\int_\dR F^2(\sqrt n x)\ell_t(x)dx \stackrel{a.s.}{\to} \|F\|^2\ell_t(0) = A(t)$,
uniformly over compact time sets, where $\ell$ is the local time of Brownian motion $\tilde B$, and
$\disp
 \|F\|^2 = \int_\dR F^2(x)dx =
 -2\int_{-\infty}^\infty \int_{-\infty}^\infty |y-z| V(y)V(z)dzdy$.
As a result, $(A_n, \tilde B)\stackrel{\mathcal{C}}{\rightsquigarrow} (A,\tilde B)$, and $A$ is clearly $\dF_{\tilde B}$-measurable.  Using Theorem \ref{thm:clt1}, both $M_n$
and $n^{-\frac{1}{4}} \int_0^{n\cdot} V(B_s)ds$ converge weakly to $W\circ A$,
where $W$ is a Brownian motion  independent of $A$ and $\tilde B$.
In fact, the full consequence of Theorem \ref{thm:clt1} states that $(M_n, A_n,\tilde B) \stackrel{\mathcal{C}}{\rightsquigarrow} (M,A,\tilde B)$ holds.
The result for $M_n$ was first proven in
\citet{Papanicolaou/Stroock/Varadhan:1977}. The proof given here is much easier and is similar to the one in \cite{Ikeda/Watanabe:1989}.
The argument above for handling $A_n$ under $\mu_V=0$, carries through
to yield $n^{-\frac{1}{2}} \int_0^{n\cdot} V(B_s)ds \stackrel{\mathcal{C}}{\rightsquigarrow} \mu_V\ell_\cdot(0)$ when $\mu_V\not=0$.

\begin{rem}
Recall that $\ell_t(0)/\sqrt{t} \stackrel{Law}{=}|B_1|$ which has Mittag-Leffler distribution with parameter $\frac{1}{2}$. Next, the inverse local time $\tau$ is known to be a L\'evy process with density $ t y^{-\frac{3}{2}}e^{-\frac{t^2}{2y}}$, $y>0$, and Laplace transform $E\left[e^{-\lambda \tau_t}\right] = e^{-t\sqrt{2\lambda}}$, $\lambda\ge 0$ \citep{Borodin/Salminen:2002}. Hence $\tau_t$  has a positive stable distribution with index $\frac{1}{2}$.
\end{rem}

\subsection{Occupation times for random walks}\label{ssec:RW}
Let $S_n$ be the symmetric simple random walk on the integers $\dZ$,
$\disp N_n(x) = \sum_{k=1}^n \I(S_k=x)$ the number of its visits to $x\in\dZ$ up to time $n$ and
$V$ a real-valued function on $\dZ$ with compact support but $V\not\equiv 0$. Setting $\disp \mu_V := \sum_{x\in\dZ} V(x)$,
\cite{Dobrushin:1955} proved that if $\mu_V\neq 0$, then $n^{-\frac{1}{2}}\sum_{k=1}^n V(S_k) \stackrel{Law}{\to} \mu_V {\cV}$,
where $\cV \stackrel{Law}{=}|Z|$, where $Z\sim N(0,1)$; while if $\mu_V=0$, then
$n^{-\frac{1}{4}}\sum_{k=1}^n V(S_k)\stackrel{Law}{\to}  \mu_v \sqrt{\cV}Z$, where $Z\sim N(0,1)$ is independent of $\cV$,
and $\disp \mu_v = 2c_V^2 - \sum_{x\in\dZ}V^2(x)$, where
$\disp c_V^2 = -\sum_{y,z\in\dZ}|y-z|V(y)V(z)=2\sum_{z\in\dZ}\left\{\sum_{y<z}V(y)\right\}^2$.
Note that $\mu_v$ corresponds to expression $\|V\|^2$ in \cite{Lee/Remillard:1994a}.

Just as in Section \ref{ssec:BM}, we prove that $(V_n,B_n) \stackrel{\cC}{\rightsquigarrow} (M,B)$ ensues when $\mu_V=0$,
where $V_n(t) := n^{-\frac{1}{4}}\sum_{k=1}^\nt V(S_k)$,
$B_n(t) := n^{-\frac{1}{2}}S_{\nt}$, $B$ is a Brownian motion,
$A(t) = c_V^2 \ell_t(0)$ with $\ell$ the local time for $B$,
and $M = W\circ A$, where $W$ is a Brownian motion independent of $A$ and $B$.
We first build the pair $(M_n,A_n)$. To this end, set $\disp G(x) = -\sum_{y\in\dZ}|x-y|V(y)$. Then
\begin{eqnarray*}
TG(x) &=& \frac{G(x+1)+G(x-1)}{2} = -\sum_y V(y)\left\{\frac{|y-x-1|+|y-x+1|}{2}\right\}\\
&=& -V(x)-\sum_{y>x} V(y)(y-x)- \sum_{y<x} V(y)(x-y)= -V(x)+G(x),
\end{eqnarray*}
with $G$ is constant outside the support of $V$: in fact, if $V\equiv 0$ on $[a,b]^\complement$, then $G(x) = \sum_{y\in [a,b]} y V(y)=c$ if $x>b$,
while if $x<a$, then $G(x) = -c$.
Consequently, $M_n(t) := n^{-\frac{1}{4}}\sum_{k=1}^\nt \{G(S_k)-TG(S_{k-1})\}$
$ \disp =
  n^{-\frac{1}{4}}\left\{G(S_\nt)-G(0)\right\}+ n^{-\frac{1}{4}}\sum_{k=1}^\nt V(S_{k-1})$
is a martingale with
\begin{eqnarray*}
A_n(t) = \langle M_n \rangle_t &=& n^{-\frac{1}{2}}\sum_{k=1}^{\nt} E \left[ \left\{G(S_k)-TG(S_{k-1})\right\}^2 |\cF_{k-1}\right] \\
&=& n^{-\frac{1}{2}} \sum_{k=1}^{\nt}\left[  T G^2(S_{k-1})-\{TG(S_{k-1})\}^2 \right] = n^{-\frac{1}{2}} \sum_{k=1}^\nt v(S_{k-1}).
\end{eqnarray*}
Now, since $TG = G-V$, it follows that $v = TG^2 - (TG)^2 = 2VG - V^2 +TG^2 -G^2$ which has compact support
since $G^2$ is constant outside $[a,b]$. As a result,
\begin{eqnarray*}
A_n(t) &=& n^{-\frac{1}{2}}\sum_{k=1}^\nt v(S_{k-1}) = n^{-\frac{1}{2}}\sum_{x\in \dZ} v(x)N_{\nt-1}(x) \\
&= & n^{-\frac{1}{2}}\sum_{x\in \dZ} v(x)\left\{ N_{\nt-1}(x)-N_{\nt-1}(0)\right\} + \mu_v n^{-\frac{1}{2}}N_{\nt-1}(0)\\
&=& n^{-\frac{1}{4}}O_P(1)  +\mu_v  n^{-\frac{1}{2}} N_{\nt}(0),
\end{eqnarray*}
using Dobrushin's results and the fact that $v$ has compact support, where
$$
\mu_v = \sum_{x\in\dZ} v(x)=-2\sum_{y,x\in\dZ}|y-x|V(y)V(x)-\sum_{x\in\dZ}V^2(x) = \sum_{x\in\dZ}\left\{V(x)+2H(x)\right\}^2,
$$
and $H(x) = \sum_{y<x}V(y)$. These expressions are proven in Appendix \ref{app:muv}. In particular, if $V(a)=1$, $a\neq 0$, and $V(0)=-1$, then $\mu_v = 4|a|-2$.
It then follows that $(A_n,B_n) \stackrel{\cC}{\rightsquigarrow} (A,B)$,
where $B$ is a Brownian motion and $A(t) = \mu_v \ell_t(0)$ is the local time of the Brownian motion at $0$.
Next, setting $f(x)=x$, one gets that
\begin{eqnarray*}
\langle M_n, B_n \rangle_t
&=& n^{-\frac{3}{4}}\sum_{k=1}^\nt \{T(fG)(S_{k-1}) - Tf(S_{k-1})TG(S_{k-1})\}  = n^{-\frac{3}{4}}\sum_{k=1}^\nt g(S_{k-1}),
\end{eqnarray*}
where $g=T(fG)-Tf TG$. Since $Tf=f$, $G(x) = c$, $x>b$ and $G(x)=-c$, $x<a$, it follows that $g$ also has compact support and, from the previous calculations
and Dobrushin's result, that for any $t\ge 0$, $ \langle M_n, B_n \rangle_t = O\left(  n^{-\frac{1}{4}} \right)$.
Therefore, using Theorem \ref{thm:clt1}, $(M_n,A_n,B_n) \stackrel{\cC}{\rightsquigarrow} (W\circ A, A,B)$,
where $W$ is a Brownian motion independent of $A$ and $B$;
hence $(V_n,A_n,B_n) \stackrel{\cC}{\rightsquigarrow} (W\circ A,A,B)$ as well, since the above calculations yield also
$\disp \sup_{t\in[0,T]}|V_n(t) - M_n(t)| =  O\left(  n^{-\frac{1}{4}} \right)$ for any $T>0$.
The corresponding result $n^{-\frac{1}{4}}V_n \stackrel{\cC}{\rightsquigarrow} \mu_V \ell$ when $\mu_V\not=0$ ensues from the
definition of the local time for Brownian motion, as in the proof of $A_n \stackrel{\cC}{\rightsquigarrow} A$ in the case $\mu_V=0$.


\subsection{Scaling limit for the random comb}\label{ssec:Comb}
Let $(\xi_n,\zeta_n)$ and $(0,\psi_n)$  be two iid sequences independent of each other, with $P(\psi_n=\pm 1)=\frac{1}{2}$, and
$P\{(\xi_{n},\zeta_{n}) = (\pm 1, 0)\}  = P\{(\xi_{n},\zeta_{n}) = (0,\pm 1)\}=\frac{1}{4}$.
The comb process $(C_1,C_2)$ \citep{Bertacchi:2006, Csaki/Csorgo/Foldes/Revesz:2009b}  is a martingale with values on the integer lattice $\mathbb{Z}^2$,
as well as a Markov chain, started at $(0,0)$ and defined by
\begin{eqnarray*}
C_1(n+1) &=& C_1(n) + \xi_{n+1}\I_{\{C_2(n)=0\}},\\
C_2(n+1) &=& C_2(n) + \psi_{n+1}\I_{\{C_2(n)\neq 0\}}+
\zeta_{n+1}\I_{\{C_2(n)=0\}}.
\end{eqnarray*}
Note that $C_2$ is also a Markov chain on its own, while $C_1$ is not. For all $n\ge 0$,
\begin{eqnarray*}
E\left[\{C_1(n+1)-C_1(n)\}^2|\cF_n\right]& =&
\frac{1}{2}\I_{\{C_2(n)=0\}},\\
E\left[\{C_2(n+1)-C_2(n)\}^2|\cF_n\right]& = &1- \frac{1}{2}\I_{\{C_2(n)=0\}},\\
E\left[\{C_1(n+1)-C_1(n)\}\{C_2(n+1)-C_2(n)\}|\cF_n\right]& =&0.
\end{eqnarray*}
Set $A_1(n) = \sum_{k=1}^n \I_{\{C_2(k)=0\}}$ and $\tau_k = \inf\{n\ge 1; A_1(n)\ge k\}$, the time of the $k$-th
visit to $0$ by $C_2$ after the initial departure, with defaults $A_1(0) = 0$ and $\tau_0=0$.
Since the increments $\sigma_n = \tau_n-\tau_{n-1}$ are iid and
$n^{-\frac{1}{2}}C_2([n\cdot]) \stackrel{\mathcal{C}}{\rightsquigarrow} W_2$, a Brownian motion, there ensues
$n^{-\frac{1}{2}}2^{-1}A_1([n\cdot]) \stackrel{\mathcal{C}}{\rightsquigarrow} \eta_2$, the local time at $0$ of $W_2$.
Build the $\dR^2$-valued martingale $\Xi_n(t):=\left\{n^{-\frac{1}{4}}C_1(\nt), n^{-\frac{1}{2}}C_2(\nt)\right\}$ with
predictable quadratic variations $n^{-\frac{1}{2}}2^{-1}A_1(\nt)$ and $n^{-1}\nt-(2n)^{-1}A_1(\nt)$ componentwise.
Since  $J_T(\Xi_n) n^{-\frac{1}{4}}\to0$ as $n\to\infty$ almost surely,
the only possible weak limits of $\Xi_n$ have continuous trajectories.
Further, $\left\{n^{-\frac{1}{2}}2^{-1}A_1([n\cdot]),n^{-1}[n\cdot]-(2n)^{-1}A_1([n\cdot])\right\}
\stackrel{\mathcal{C}}{\rightsquigarrow} \{t\mapsto(\eta_2(t),t)\}$, a process with continuous trajectories.
Since $\Xi_n$ has uncorrelated components, all the conditions of Theorem \ref{thm:clt1} are met
and $\Xi_n \stackrel{\mathcal{C}}{\rightsquigarrow} (W_1\circ\eta_2,W_2)$.
Notice that if we write $X_n = C_1(\tau_n)-C_1(\tau_{n-1})$, then $(X_n,\sigma_n)_{n\ge 1}$ are iid.


\section{Volatility modeling and estimation}\label{sec:volatility_stats}

The financial return $R(t)$ at time $t$ for some investment under consideration is modelled as $R(t) = Z\circ \tau (t)$, with
$Z(s) := \gamma B(s)+ \beta s$, where $\tau$ is a strictly increasing process, independent of the Brownian motion $B$. This model was first proposed by \citet{Ane/Geman:2000}.
Time scale $\tau$ is meant to reflect business cycles and other features known collectively in economics as business time;
$Z$ is thus the financial return adjusted accordingly. Constant $\gamma$ is a scaling parameter and constant $\beta$ the trend after
the correction for business time.   Note that a different model with the same distributional features is $R(t) = \int_0^t \sigma(s) dW(s)+ \beta \tau(t)$, where $\sigma$ is a continuous process independent of the Brownian motion $W$ and $\tau(t) = \int_0^t \sigma^2(s)ds$. The latter was considered by \cite{Barndorff-Nielsen/Shephard:2002}.

Rigorous treatment of these models requires some technical results concerning
$D$-valued square integrable $\mathbb{F}$-martingales $M$ and their compensator $A=\langle M \rangle$.
We gather these next.
In what follows, we consider the partitions  $s_{k,\delta}=k\delta$
independent of $t$,  where $K(t,\delta)  = \left\lfloor \frac{t}{\delta}\right\rfloor$.
For each $\delta>0$, construct
$$
V_{p,\delta}(t) := \delta^{1-\frac{p}{2}} \sum_{1\le k\le K(t,\delta)}|M(s_{k,\delta})-M(s_{k-1,\delta})|^p
$$
and
$$
U_{p,\delta}(t) := \delta^{1-\frac{p}{2}} \sum_{1\le k\le K(t,\delta)}\{A(s_{k,\delta})-A(s_{k-1,\delta})\}^{\frac{p}{2}}.
$$
Before stating the next result, 
set $\disp
\mu_{p} = E(|Z|^p) = 2^{\frac{p}{2}} \dfrac{\Gamma\left(\frac{p+1}{2}\right) }{\Gamma\left(\frac{1}{2}\right)},
$
where $Z\sim N(0,1)$. So $\mu_{2n} = \prod_{k=1}^n (2k-1)$ and $\mu_{2n+1} = \sqrt{\frac{2}{\pi}}\cdot\prod_{k=1}^n (2k)$.
Typically useful values are $\mu_2=1$, $\mu_4=3$ and $\mu_8=105$.

\begin{lem}\label{lem:pvariation2}
Given is a real-valued martingale $M$ started at $M_0=0$ with finite $p^{th}$ moment for some integer $p\ge2$
and compensator $A$, where  $A_t=\int_0^ta_sds$, for some non-negative and continuous stochastic process $a$. Assume the existence of a Brownian motion $B$ such that $M = B\circ A$, with $A$ independent of $B$.
\BB{Then, for $0\le s\le t<\infty$, we have}
\begin{equation}\label{eq:Lp}
E\left\{|M_t-M_s|^{p} |\cF_s \right\} = \mu_{p} E\left\{ (A_t-A_s)^{\frac{p}{2}} |\cF_s \right\}.
\end{equation}
\BB{In addition,} $\disp \lim_{\delta\downarrow 0} U_{p,\delta} = \cU_{p}$ a.s., where $\disp \cU_p(t) = \int_0^t a_s^{\frac{p}{2}} ds$.
Furthermore,  under the additional assumption that $M$ has finite $(2p)^{th}$ moment,
$\disp
V_{p,\delta} \stackrel{Pr}{\longrightarrow}\cV_{p}=\mu_p \cU_{p}$ as $\delta\downarrow 0$.
In particular, if $M_t = \int_0^t \sigma_s dW_s$, for some Brownian motion $W$ and continuous non-negative process $\sigma$ independent of $W$, then $\disp \cV_{p}(t)= \mu_p \int_0^t \sigma_s^p ds$.
\end{lem}

\begin{proof}
Equation \eqref{eq:Lp} proceeds from
$E\left\{|B\circ A_t-B\circ A_s|^{p} |\sigma\{A\}\right\} = \mu_{p} (A_t-A_s)^{\frac{p}{2}}$ which itself ensues from
the independence of $A$ and $B$, plus the moments of a standard normal distribution.
When $p\ge 2$, note that
$\disp \cU_p(t) = \int_0^t a_s^{\frac{p}{2}} ds \ge U_{p,\delta}(t)$ and
$\disp
  \delta^{1-\frac{p}{2}}\{A(s_{k,\delta})-A(s_{k-1,\delta})\}^{\frac{p}{2}}  = \delta \{a(w_{k,\delta})\}^{\frac{p}{2}}
\le
\int_{s_{k-1,\delta}}^{s_{k,\delta}} a_s^{\frac{p}{2}}ds$,
for some $w_{k,\delta}\in[s_{k-1,\delta},s_{k,\delta}]$.
As a result,  $\disp\lim_{\delta\downarrow 0} U_{p,\delta}=\cU_p$ since $a$ is continuous and  $\disp  \lim_{\delta\downarrow 0} \sum_{k=1}^{K(t,\delta)}\int_{s_{k-1,\delta}}^{s_{k,\delta}} \left[ a_u^{\frac{p}{2}}  - \{a(w_{k,\delta})\}^{\frac{p}{2}}\right]du =0$.
It remains to show that the filter of martingales $Z_\delta:=V_{p,\delta}-\mu_pU_{p,\delta}$ converges in probability to $0$,
 uniformly on compact time sets --- the martingale property proceeds at once, just as in the proof of Equation \eqref{eq:Lp}.
 Assuming $E\{|M_t|^{2p}\}<\infty$ for any $t\ge0$, the representation $M = B\circ A$ yields
  \begin{multline*}
E\left\{ |M(s_{k,\delta})-M(s_{k-1,\delta})|^p \{A(s_{k,\delta})-A(s_{k-1,\delta})\}^{\frac{p}{2}} \Big|\cF_{s_{k-1,\delta}}\right\} \\
= \mu_p E\left\{  \{A(s_{k,\delta})-A(s_{k-1,\delta})\}^{p} \Big|\cF_{s_{k-1,\delta}}\right\} .
\end{multline*}
Expanding the square in
 \begin{multline*}
E\left\{\left( |M(s_{k,\delta})-M(s_{k-1,\delta})|^p
 - \mu_p \{A(s_{k,\delta})-A(s_{k-1,\delta})\}^{\frac{p}{2}} \right)^2\Big|\cF_{s_{k-1,\delta}}\right\} \\
= (\mu_{2p} - \mu_p^2) E\left\{  \{A(s_{k,\delta})-A(s_{k-1,\delta})\}^{p} \Big|\cF_{s_{k-1,\delta}}\right\}
\end{multline*}
implies that $Z_\delta$ is square integrable with compensator $\langle Z_\delta \rangle$ given by
$$
\langle Z_\delta \rangle_t = \delta^{2-p} (\mu_{2p} - \mu_p^2)  \sum_{1\le k\le K(t,\delta)}
E\left\{  \{A(s_{k,\delta})-A(s_{k-1,\delta})\}^{p} \Big|\cF_{s_{k-1,\delta}}\right\},
$$
so there ensues $E\{\langle Z_\delta \rangle_t\}=\delta (\mu_{2p} - \mu_p^2) E\{U_{2p,\delta}(t)\}$ which goes to $0$ with $\delta$.
By Lemma \ref{lem:lenglart}, $\disp \sup_{0\le s\le t}|Z_\delta(s)|$ converges in probability to $0$ with $\delta$ as well.
\end{proof}
\begin{rem}
This result is an extension of \cite{Barndorff-Nielsen/Shephard:2003a}.
They only prove their result for a fixed $t$ with convergence in probability but claimed it could also be true as a process.
The case $p=2$ is just the definition of quadratic variation $[M]$ --- see the proof of \citet[Proposition 2.3.4]{Ethier/Kurtz:1986},
where the existence of $\cV_{2}:=\lim_{\delta\downarrow 0} V_{2,\delta} $ is shown to hold for any right continuous local martingale
and without any additional restriction, neither on $A$ nor on the filter.
Note that, when $p>2$, the limit $\cV_{p}$ does not exist if $M$ is not continuous everywhere.
The asymptotics for $V_{p,\delta}$ under $p\in(0,2)$ are covered extensively by \cite{Jacod:2007} and \cite{Jacod:2008},
for a large class of processes with jumps, based on equally spaced observations.
The cases $p\ge3$ are also examined in \citet[Theorems 2.11(i)]{Jacod:2008} --- the presence of jumps in the limit yields
a CLT with $\delta^{-\frac{1}{2}}$ instead of $\delta^{1-\frac{p}{2}}$. For our continuous limits, the larger fluctuations observed there disappear.
See his comments in Remarks 2.14, 2.15 and 2.16.
  \end{rem}

We can now state a first consistency result for an estimator of the realized volatility error in investment returns,
when data is collected at regular intervals. This is an extension of the results presented in \cite{Barndorff-Nielsen/Shephard:2002}, where convergence was limited to a fixed $t$, while here we obtain the convergence of the whole process.
Before stating the first theorem, define
\begin{eqnarray*}
X_n(t) &:=& n^{\frac{1}{2}} \sum_{j=1}^{\lfloor nt \rfloor} \left[ \left\{\Delta_n M\left(\frac{j}{n}\right) \right\}^2 -\Delta_n A\left(\frac{j}{n}\right)  \right],\\
\langle X_n \rangle _t
&:=& 2n \sum_{j=1}^{\lfloor nt \rfloor} E\left[\left.  \left\{\Delta_n A\left(\frac{j}{n}\right) \right\}^2 \right| \cF_{\frac{j-1}{n}}\right],
\\
V_n(t) &:=& n \sum_{j=1}^{\lfloor nt \rfloor} \left\{\Delta_n M\left(\frac{j}{n}\right) \right\}^4,
\end{eqnarray*}
where $\Delta_n f(s) = f(s)-f(s-1/n)$. 

\begin{thm}[Numerical scheme]\label{thm:num_scheme}
Assume that both $A(t) \to\infty$ and $\cV_{4}(t) \to\infty$ as $t\to\infty$.
Under all the conditions of Lemma \ref{lem:pvariation2} with $p=4$, including the finiteness of $E\{|M_t|^8\}<\infty$ for any $t\ge0$,
there is a standard Brownian motion $W$ independent of $\cV_{4}$ such that
$(X_n,\langle X_n \rangle,V_n) \stackrel{\mathcal{C}}{\rightsquigarrow} ( W\circ \cA ,\cA,\cV_{4})$,
where $\cA = 2\cU_4 = \dfrac{2}{3}\cV_{4}$.
Furthermore, for any adapted $D$-valued process $N$
such that $\disp n^{\frac{1}{2}} \sum_{j=1}^{\lfloor nt \rfloor} \left\{\Delta_n N\left(\frac{j}{n}\right) \right\}^2 \stackrel{Pr}{\to} 0$ and
$\disp n \sum_{j=1}^{\lfloor nt \rfloor} \left\{\Delta_n N\left(\frac{j}{n}\right)\right\}^2\Delta_n A\left(\frac{j}{n}\right) \stackrel{Pr}{\to} 0$
both hold for all $t>0$, there also comes $Y_n \stackrel{\mathcal{C}}{\rightsquigarrow} W\circ \cA$, where
$\disp Y_n(t):= n^{\frac{1}{2}} \sum_{j=1}^{\lfloor nt \rfloor} \left[ \left\{\Delta_n (M+N)\left(\frac{j-1}{n}\right) \right\}^2 -\Delta_n A\left(\frac{j-1}{n}\right)\right]$.
\end{thm}
\begin{proof}
By Lemma \ref{lem:pvariation2}, $V_n$ converges to $\cV_4$ in probability and hence
$V_n \stackrel{\mathcal{C}}{\rightsquigarrow} \cV_4$ holds, by Proposition \ref{prop:Cincreasing}.
The expression for $\langle X_n \rangle$ is a direct consequence of Lemma \ref{lem:pvariation2}, which entails
$E\left[ \left\{\Delta_n M\left(\frac{j}{n}\right) \right]^4 | \cF_{\frac{j-1}{n}}\right]=3E\left[ \left\{\Delta_n A\left(\frac{j}{n}\right) \right]^2 | \cF_{\frac{j-1}{n}}\right]$,
while the representation $M = B\circ A$ yields
$$
E\left[ \left\{\Delta_n M\left(\frac{j}{n}\right) \right]^2 \Delta_n A\left(\frac{j}{n}\right)  | \cF_{\frac{j-1}{n}}\right]= E\left[ \left\{\Delta_n A\left(\frac{j}{n}\right) \right]^2 | \cF_{\frac{j-1}{n}}\right].
$$
Also, $V_n = \frac{3}{2}\langle X_n \rangle +  \cZ_n$, where
$$
\cZ_n(t)=n \sum_{j=1}^\nt \left[ \left\{ \Delta_n M\left(\frac{j}{n}\right)\right\}^4 - E\left[\left\{\Delta_n M\left(\frac{j}{n}\right)\right\}^4 | \cF_{\frac{j-1}{n}} \right] \right]
$$
is a martingale with
$$
\langle \cZ_n \rangle_t = \mu_8 n^2 \sum_{k=1}^\nt E\left[ \left\{\Delta_n A\left(\frac{j}{n}\right) \right]^4 | \cF_{\frac{j-1}{n}}\right]
- \mu_4^2 n^2  \sum_{k=1}^\nt E^2\left[ \left\{\Delta_n A\left(\frac{j}{n}\right) \right]^2 | \cF_{\frac{j-1}{n}}\right].
$$
By Lemma \ref{lem:pvariation2} with $p=8$, $nE\{\langle \cZ_n \rangle_t\}$ is bounded above by sequence
$E\{U_{8,1/n}(t)\}$ which converges to $E\{\cU_8(t)\}$ and hence
$\langle \cZ_n \rangle \stackrel{\mathcal{C}}{\rightsquigarrow} 0$ holds, by Proposition \ref{prop:Cincreasing}.
By Lemma \ref{lem:lenglart}, $\disp \sup_{0\le t\le T}|\cZ_n(t)|$ converges in probability to $0$.
This implies $\langle X_n \rangle \stackrel{\mathcal{C}}{\rightsquigarrow} 2 \cU_4 = \dfrac{2}{3}\cV_{4}$.
Next, let $Z_j$ be iid standard Gaussian random variables independent of $A$,
set $\disp \dB_n(t) = (2n)^{-\frac{1}{2}}\sum_{j=1}^\nt (Z_j^2-1)$;
further set $a_{n,j} = A\left(\frac{j}{n}\right)-A\left(\frac{j-1}{n}\right)$.
It is then clear that $(A_n,\dB_n) \stackrel{\cC}{\longrightarrow} (A,\dB)$, where $\dB$ is a Brownian motion independent of $a$.
For any  $0=t_0<t_1<\cdots < t_m$, and $\lambda_1,\ldots, \lambda_m \in \dR$,
since $M=B\circ A$, where $B$ is a Brownian motion independent of $A$,
one has
\begin{multline*}
E\left[\exp\left[i \disp\sum_{k=1}^m   \lambda_{k}\{X_n(t_k)-X_n(t_{k-1})\} \right] \right] \\
=E\left[\exp\left[i \sn \sum_{k=1}^m  \lambda_{k}\sum_{j=\lfloor nt_{k-1}\rfloor+1}^{\lfloor nt_{k}\rfloor} a_{n,j} \left(Z_j^2-1\right)  \right]\right]\\
=E\left[\exp\left[i 2^{\frac{1}{2}} \sni \sum_{k=1}^m  \lambda_{k}\sum_{j=\lfloor nt_{k-1}\rfloor+1}^{\lfloor nt_{k}\rfloor} a\left(\frac{j-1}{n}\right) \left(\frac{Z_j^2-1}{2^{\frac{1}{2}}}\right)  \right]\right]+o(1).\\
\end{multline*}
One can then use Proposition \ref{prop:stochInt} to conclude that
$(X_n,\langle X_n \rangle,V_n, \dB_n) \stackrel{\mathcal{C}}{\rightsquigarrow} (X,\cA,\cV_{4},\dB)$,
where $\dB$ is a Brownian motion independent of $\cA = 2\cU_4$ and $\disp X_t = \int_0^t a(s)d\dB(s)$.
Furthermore, $X$ can be written as $X=W\circ \cA$, where $W$ is a Brownian motion independent of $\cA$. Finally, writing
$\disp
Z_n(t):= n^{\frac{1}{2}} \sum_{j=1}^{\lfloor nt \rfloor} \Delta_n N\left(\frac{j-1}{n}\right)$
yields square integrable martingale
$\disp
\int_0^t Z_n(s) dM(s) = n^{\frac{1}{2}} \sum_{j=1}^{\lfloor nt \rfloor} \Delta_n M\left(\frac{j}{n}\right) \Delta_n N\left(\frac{j-1}{n}\right)
$
with
$\disp \int_0^t Z_n^2(s)ds = \sum_{j=1}^{\lfloor nt \rfloor} \left\{\Delta_n N\left(\frac{j-1}{n}\right)\right\}^2$
and quadratic variation
$$
\left\langle \int_0^t Z_n(s) dM(s) \right\rangle = \int_0^t Z_n^2(s) d\langle M \rangle_t
= n \sum_{j=1}^{\lfloor nt \rfloor} \{\Delta_n N\left(\frac{j-1}{n}\right)\}^2\Delta_n A\left(\frac{j-1}{n}\right).
$$
Therefore
$Y_n(t)-X_n(t)= n^{\frac{1}{2}}\int_0^t Z_n^2(s)ds + 2\int_0^t Z_n(s) dM(s) \stackrel{Pr}{\to} 0$
holds for every $t>0$, since $\int_0^t Z_n^2(s) d\langle M \rangle_t \stackrel{Pr}{\to} 0$.
Hence $Y_n-X_n \stackrel{\mathcal{C}}{\rightsquigarrow} 0$ holds as well, yielding
the last statement for $Y_n$ via Lemma \ref{lem:lenglart}, Theorem \ref{thm:M1tightness} and Proposition \ref{prop:Cincreasing}.
\end{proof}
A frequent choice for perturbation process $N$ is a linear function of the volatility. Recall the modulus of continuity of $A$, defined by
$$
\omega_\cC(A,\delta,T) = \sup_{0 \le t_1 < t_2 \le T, \; t_2-t_1<\delta}\|A(t_2)-A(t_1)\|.
$$
\begin{cor}\label{cor:num_scheme2}
Suppose that $N(t)=\mu t+\beta A(t)$, for some constants $\mu\in\dR$ and $\beta\ge0$, and that
the modulus of continuity of $A$ satisfies $\delta^{-\alpha}\omega_\cC(A,\delta,t) \stackrel{Pr}{\to} 0$ as $\delta\to0$,
for some $\alpha>\frac{3}{4}$ and all $t>0$. The conclusions of Theorem \ref{thm:num_scheme} hold.
\end{cor}

\begin{proof}
For the terms in $A$, the two conditions on $N$ are a consequence of
 \begin{multline*}
\disp n^{\frac{1}{2}} \sum_{j=1}^{\lfloor nt \rfloor} \left\{\Delta_n A\left(\frac{j}{n}\right)\right\}^2
+ n \sum_{j=1}^{\lfloor nt \rfloor} \left\{\Delta_n A\left(\frac{j}{n}\right)\right\}^3
\le tn^{3/2}\omega_\cC^2(A,1/n,t) + tn^2\omega_\cC^3(A,1/n,t).
\end{multline*}
The terms in $\mu$ are treated similarly.
\end{proof}

This numerical scheme extends readily to higher powers. We state it without proof, leaving the details to the reader. First, define
\begin{eqnarray*}
X_{n,p}(t) &:=& n^{\frac{p-2}{4}} \sum_{j=1}^{\lfloor nt \rfloor} \left[ \left|\Delta_n M\left(\frac{j}{n}\right) \right|^{\frac{p}{2}} -\mu_{\frac{p}{2}} \{\Delta_n A\left(\frac{j}{n}\right)\}^{p/4}  \right],\\
\langle X_{n,p} \rangle_t &:=& \left(\mu_{p}-\mu_{\frac{p}{2}}^2\right) n^{\frac{p}{2}-1}
\sum_{j=1}^{\lfloor nt \rfloor} E\left[ \left\{\Delta_n A\left(\frac{j}{n}\right) \right\}^{\frac{p}{2}} | \cF_{\frac{j-1}{n}}\right],\\
V_{n,p}(t) &:=& n^{\frac{p}{2}-1} \sum_{j=1}^{\lfloor nt \rfloor} \left|\Delta_n M\left(\frac{j}{n}\right) \right|^{p}.
\end{eqnarray*}

\begin{thm}[New numerical scheme]\label{thm:num_scheme_new}
Assume that $A(t) \to\infty$ and $\cV_{p}(t) \to\infty$ as $t\to\infty$, with $\cV_{p}$ continuous, for some $p>4$.
Under all the conditions of Lemma \ref{lem:pvariation2}, including the finiteness of the $(2p)^{th}$ moment of $M$,
there is a standard Brownian motion $W$ independent of $\cV_{p}$ such that
$(X_{n,p},\langle X_{n,p}\rangle,V_{n,p}) \stackrel{\mathcal{C}}{\rightsquigarrow}
( W\circ \cA_p, \cA_p, \cV_{p})$,
where  $\cA_p = \dfrac{\left(\mu_{p}-\mu_{\frac{p}{2}}^2\right)}{\mu_{p}}\cV_{p}$.

\end{thm}

\begin{rem}\label{rem:BNS}
In their analysis of the asymptotic properties of realized volatility error in investment returns,
\cite{Barndorff-Nielsen/Shephard:2002} assumed that $M_t = \int_0^t \sigma_u dW_u$,
with square integrable $\sigma$ independent of Brownian motion $W$ ---
$\sigma_u^2$ is known as the spot volatility or instantaneous volatility at time $u$.
Since $A_t = \int_0^t \sigma_u^2 du$ is continuous, it follows
from the proof of Proposition \ref{prop:stochInt} that $M$ can also be written as $M=B\circ A$,
where $B$ is a Brownian motion independent of $A$.
They also said that one could consider $N(t)=\mu t+\beta A(t)$ but mentioned that it could be difficult to obtain $p$-variation convergence, specially if $\beta\neq 0$. From our Corollary \ref{cor:num_scheme2}, we see that having this extra term
does not influence the value of the limit $W\left(\dfrac{2}{3}\cV_{4}\right)$ for $Y_n$,
confirming that they can be ignored when estimating realized volatility,
at least below the order of third moments. Note that their setting is a modification of the models introduced in \citet{Ane/Geman:2000}
by adding the drift term $\mu$, while setting business time to $\tau=A$ and scaling parameter to $\gamma=1$.
Note also that from Theorem \ref{thm:num_scheme}, \citet[Theorem 1]{Barndorff-Nielsen/Shephard:2002} can be restated as follows: as $n\to\infty$,
$\disp
\dfrac{X_n(1)}{ \left\{2n\sum_{j=1}^n \{\Delta_n A\left(\frac{j}{n}\right)\}^2\right\}^{\frac{1}{2}}} \stackrel{Law}{\to } N(0,1)$
and $\disp n\sum_{j=1}^n \{\Delta_n A\left(\frac{j}{n}\right)\}^2 \stackrel{a.s.}{\to } \int_0^1 \sigma^4(s)ds := \cU_{4}(1)$.
Furthermore, $\disp 2 n\sum_{j=1}^\nt \{\Delta_n A\left(\frac{j}{n}\right)\}^2  - \langle X_n \rangle _t
$ is a martingale that converges to $0$, due to the continuity of $A$ and the fact that $\sigma^4$ is locally integrable.
As a result, $\cV_{4} =3 \cU_{4}$.
Their result is therefore an upshot of Theorem \ref{thm:num_scheme}.
More generally, \citet{Barndorff-Nielsen/Graversen/Jacod/Podolskij/Shephard:2006}
prove a CLT for sequences of continuous $\dR^d$-valued semimartingales $Y$
of the following general form --- we stick to the case $d=1$ here for the sake of simplicity:
$$
Y(t)=Y(0) + \int_0^t a_sds + \int_0^t \sigma_{s-} dW(s),
$$
with $W$ a standard Brownian motion, $a$ bounded predictable and $\sigma$ $D$-valued.
For any pair $G$ and $H$ of continuous real-valued functions with at most polynomial growth, the sequence of approximations
$$
X_n(G,H)_t = n^{-1} \sum_{j=1}^{\lfloor nt \rfloor} G\left\{\frac{1}{2}\Delta_n Y\left(\frac{j}{n}\right) \right\}\cdot H\left\{\frac{1}{2}\Delta_n Y((j+1)/n) \right\}
$$
are first shown to obey a Law of Large Numbers :
$$
X_n(G,H)_t \stackrel{Pr}{\to} X(G,H)_t :=\int_0^t \rho_{\sigma_s}(G)\rho_{\sigma_s}(H)ds
$$
where $\rho_{\sigma_s}(G):=E\{G(Z)\}$ where $Z\sim N(0,\sigma_s)$. Under some additional restrictions on both
the stochastic structure of $\sigma_s$ and the smoothness of $G$ and $H$, a CLT also ensues --- actually in the sense of stable convergence:
$$
\frac{1}{2}\{X_n(G,H)_t-X(G,H)_t\} \stackrel{Law}{\to} U(G,H)_t
$$
where $U(G,H)_t$ is a stochastic integral with respect to another Brownian motion independent of the ambient filtration.
A functional version of this result should ensue from our Theorem \ref{thm:num_scheme_new} through the same type of arguments,
when  $G$ and $H$ have at most polynomial growth. We do not pursue this here.
\end{rem}

\section{Conclusion}

We have shown that under  weak conditions involving compensators, one can get a CLT for general martingales, and the limiting process is a mixture of a Brownian motion with the limiting compensator of the sequence of martingales. These conditions are easy to verify and are general enough to be applicable to a wide range of situations.

\appendix


\section{Auxiliary results}\label{app:muv}


\subsection{Some useful results}

\begin{lem}[Lenglart's inequality]\label{lem:lenglart}
Let $X$ be an $\dF$-adapted $D$-valued process. Suppose that $Y$ is optional, non-decreasing,
and that, for any bounded stopping time $\tau$, $E|X(\tau)|\leq E\{Y(\tau)\}$.
Then for any stopping time $\tau$ and all $\varepsilon,\eta >0$,
\begin{itemize}
\item[{a)}]
if $Y$ is predictable,
\begin{equation}\label{eng1}
P(\sup_{s\leq \tau}|X(s)|\geq\varepsilon) \leq \frac{\eta}{\varepsilon} +
P(Y(\tau)\geq\eta).
\end{equation}

\item[b)]
if $Y$ is adapted,
\begin{equation}\label{eng2}
P(\sup_{s\leq \tau}|X(s)|\geq\varepsilon) \leq \frac{1}{\varepsilon}
\left[\eta+ E\left\{J_\tau(Y) \right\}\right] +
P(Y(\tau)\geq\eta).
\end{equation}
\end{itemize}
\end{lem}
\begin{proof} See \citet[Lemma I.3.30]{Jacod/Shiryaev:2003}.
\end{proof}

\vfil

Proving $\cJ_1$-tightness generally involves the following lemma.

\begin{lem}[Aldous's criterion]\label{lem:aldous}
Let $\{X_n\}_{n\ge 1}$ be a sequence of $D$-valued processes. Suppose that for any sequence of bounded discrete stopping times
${\left\{{\tau_n}\right\}_{n\ge 1}}$ and for any sequence ${\left\{{\delta_n}\right\}_{n\ge 1}}$ in $[0,1]$ converging to $0$,
the following condition holds, for every $T>0$:
(A) $X_n((\tau_n+\delta_n)\wedge T)- X_n(\tau_n)\stackrel{Law}{\to} 0$. Then, ${\left\{{X_n}\right\}_{n\ge 1}}$ is $\cJ_1$-tight,
if either of the two following conditions holds:
\begin{enumerate}
\item $\{X_n(0)\}_{n\ge 1}$ and $(J_T(X_n))_{n\ge 1}$ are tight;
\item $\{X_n(t)\}_{n\ge 1}$ is tight for any $t\in [0,T]$.
 \end{enumerate}

\end{lem}

\begin{proof} See \citet{Aldous:1978}, \citet[Theorem VI.4.5]{Jacod/Shiryaev:2003}
or for several variants see \citet[Theorem 3.8.6]{Ethier/Kurtz:1986}.
Note that condition (1) or condition (2) are  necessary for $\cJ_1$-tightness, but not condition (A).
\end{proof}

Now comes the main result about tightness, stated for real-valued processes. Before stating it,
set $J(x_1,x_2,x_3) = |x_2-x_1|\wedge |x_2-x_3|$, where $x\wedge y =
\min(x,y)$. For $x \in D$, set
$$
w_{\cJ_1}(x,\delta,T) = \sup_{0 \le t_1 < t_2 < t_3\le T, \; t_3-t_1<\delta}
J\{x(t_1),x(t_2),x(t_3)\}.
$$
It follows from \cite[Remarques:I.6]{Rebolledo:1979} and \cite[VII, Lemma 6.4]{Parthasarathy:1967}
that, for any $T>0$ and $\delta>0$,
\begin{equation}\label{eq:partha}
  J_T(x) \le \omega_\cC(x,\delta,T) \le  J_T(x) + 2\omega_{\cJ_1}(x,\delta,T).
\end{equation}
\begin{rem}\label{rem:partha}
It follows from \eqref{eq:partha} that $X_n$ is $\cC$-tight iff $X_n$ is $\cJ_1$-tight and $J_T(X_n)\stackrel{Law}{\longrightarrow}0$.
For if $\epsilon>0$ is given, then, for any $\delta>0$,
$$
P(\omega_\cC(X_n,\delta,T) > \epsilon) \le P(J_T(X_n) >\epsilon/2) + P(\omega_{\cJ_1}(X_n,\delta,T)>\epsilon/4)\stackrel{n\to\infty}{\longrightarrow} 0.
$$
\end{rem}

\begin{thm}\label{thm:M1tightness}
Let $M_n$ be a sequence of $D$-valued square integrable $\mathbb{F}$-martingales with  $M_n(0)= 0$,
with quadratic variation $[M_n]$ and compensator $A_n = \langle M_n \rangle $.
\begin{itemize}
\item[a)]
Assume that for any $t>0$, $\limsup_{n\to\infty} E\{J_t^2(M_n)\}=0$ holds.
Then $M_n$ is $\mathcal{C}$-tight if and only if $[M_n]$ is $\mathcal{C}$-tight.

\item[b)]
Assume that for any $t>0$, $\limsup_{n\to\infty} E\{J_t^2(M_n)\}=0$ and $J_t(A_n) \stackrel{Law}{\to} 0$ hold.
Then the $\mathcal{C}$-tightness of $[M_n]$ implies that of $A_n$.

\item[c)]
Assume that for any $t>0$, $J_t(M_n) \stackrel{Law}{\to} 0$ hold.
Then the $\mathcal{C}$-tightness of $A_n$ implies that of both $M_n$ and $[M_n]$.

\end{itemize}
\end{thm}

\begin{rem}
If every $M_n$ is continuous, then $M_n$ is $\mathcal{C}$-tight if and only if $A_n$ is $\mathcal{C}$-tight, by statement a),
since both $J_t(M_n)=0$ and $[M_n]=A_n$ then hold --- this goes back to \citet{Rebolledo:1979}.
Without the continuity assumption on $M_n$, this equivalence no longer holds.
Note also that c) follows from \cite[Theorem 4.13]{Jacod/Shiryaev:2003}.
\end{rem}
\begin{proof}  The idea of the proof is to show $\cJ_1$-tightness, and then use Remark \ref{rem:partha}.
For statement a), suppose first that $M_n$ is $\cC$-tight. Set $ X_n(t)= [M_n]_{t+\tau_n} - [M_n]_{\tau_n} $ and $Y_n(t)=  {\sup_{0\leq s\leq t}}\{M_n (s+\tau_n)-M_n(\tau_n)\}^2$,
 where $\tau_n$ is a stopping time uniformly bounded by $T$ for any $n$. Then, for any bounded stopping time $\tau$, $E\{X_n(\tau)\} \le E\{Y_n(\tau)\}$.
Let $\delta$ be a bounded stopping time.
By \eqref{eng2} of Lemma \ref{lem:lenglart}, we have, for any $\varepsilon,\eta >0$,
\begin{multline}
 P([M_n]_{\tau_n+\delta_n}-[M_n]_{\tau_n}\geq\varepsilon) \label{eq:eng3} \\
 \le  \frac{\eta}{\varepsilon} + \frac{1}{\varepsilon} E\{J_{\delta_n}(Y_n)\} +
P(Y_n(\delta_n))\geq\eta) \\
\le \frac{\eta}{\varepsilon} + \frac{1}{\varepsilon} E\left\{J_{T+1}^2(M_n)\right\} +
P\left\{ \omega_\cC(M_n,\delta_n,T+1) > \sqrt\eta \right\}.
\end{multline}
Since $M_n$ is $\mathcal{C}$-tight, it follows that $P\left\{ \omega_\cC(M_n,\delta_n,T+1) > \sqrt\eta \right\} \to 0$
as $n\to \infty$. Set  $\eta = \varepsilon^2$. Then
$\limsup_{n\to\infty} P\left\{[M_n]_{\tau_n+\delta_n}-[M_n]_{\tau_n}\geq\varepsilon\right\}\le \varepsilon$,
showing that $[M_n]$ meets both conditions (A) and (1) of Lemma \ref{lem:aldous}, since $J_T([M_n])=J_T^2(M_n)$.
Hence $[M_n]$ is $\cJ_1$-tight. The fact that $[M_n]$ is $\cC$-tight follows from Remark \ref{rem:partha}  and $J_T([M_n])\stackrel{Law}{\longrightarrow} 0$. To complete the proof of a), assume now that $[M_n]$ is $\cC$-tight. Using Lemma \ref{lem:lenglart} yields
\begin{multline*}
P\left\{ |M_n(\tau_n+\delta_n)-M_n(\tau_n)|\geq\varepsilon \right\}  \\
 \le  \frac{\eta}{\varepsilon^2} + \frac{1}{\varepsilon^2} E\{J_{\delta_n}([M_n]_{\tau_n+\cdot}-[M_n]_{\tau_n})\} +
P\left\{[M_n]_{\tau_n+\delta_n}-[M_n]_{\tau_n}>\eta\right\} \\
\le  \frac{\eta}{\varepsilon^2} + \frac{1}{\varepsilon^2} E\{J_{T+1}([M_n])\} +
P\left\{ \omega_\cC([M_n],\delta_n,T+1) > \eta \right\}.
\end{multline*}
Choosing $\eta = \varepsilon^3$,
both conditions (A) and (1) of Lemma \ref{lem:aldous} are met and $M_n$ is $\cJ_1$-tight.   Since $J_T(M_n) \stackrel{Law}{\longrightarrow} 0$, it follows from Remark \ref{rem:partha} that $M_n$ is $\cC$-tight.
For statement b), \eqref{eq:eng3} becomes
$$
P(A_n(\tau_n+\delta_n)-A_n(\tau_n)\geq\varepsilon)   
\le\frac{\eta}{\varepsilon} + \frac{1}{\varepsilon} E\{J_{T+1}^2(M_n)\} +
P\left\{ \omega_\cC([M_n],\delta_n,T+1) > \eta \right\},
$$
showing that $A_n$ meets both conditions (A) and (1) of Lemma \ref{lem:aldous}. As a result, $A_n$ is $\cJ_1$-tight. Since $J_T(A_n) \stackrel{Law}{\longrightarrow} 0$, $A_n$ is also $\cC$-tight, using Remark \ref{rem:partha}.
Finally, for statement c), use \eqref{eng1} of Lemma \ref{lem:lenglart} instead to prove that
each of $M_n$ and $[M_n]$ meets both conditions (A) and (1) of Lemma \ref{lem:aldous}. As a result, both  $M_n$ and $[M_n]$ are $\cJ_1$-tight. Since $J_T(M_n) \stackrel{Law}{\longrightarrow} 0$,  $M_n$ is $\cC$-tight by Remark \ref{rem:partha}, and so is $[M_n]$ by a).
\end{proof}
\begin{prop}\label{prop:Cincreasing}
Let $D$-valued non-decreasing processes $A_n$ and some continuous process $A$
be such that $A_n \stackrel{f.d.d.}{\rightsquigarrow} A$. Then $A_n$ is $\mathcal{C}$-tight
and $A_n \stackrel{\mathcal{C}}{\rightsquigarrow} A$.
\end{prop}
\begin{proof} For any fixed $T>0$ and $\delta>0$ there holds, with $i$ running only through the non-negative integers,
\begin{multline*}
 \limsup_{n\to\infty} P(\omega_\cC(A_n, \delta,T) \ge \epsilon) \le  \limsup_{n\to\infty}P\left(3 \max_{0 \le i \le T/\delta }
\sup_{ i\delta \le s \le T\wedge (i+1)\delta}|A_n(s)-A_n(i\delta)| \ge \epsilon \right)\\
\le
\limsup_{n\to\infty}P\left( \bigcup_{0 \le i \le T/\delta }
\left\{ |A_n(T\wedge (i+1)\delta)-A_n(i\delta)| \ge  \epsilon/3 \right\} \right)\\
\le
P\left( \bigcup_{0 \le i \le T/\delta}\left\{ |A(T\wedge (i+1)\delta)-A(i\delta)|
\ge  \epsilon/3 \right\} \right)
\le P(\omega_\cC(A, \delta,T) \ge  \epsilon/3),
\end{multline*}
using convergence in law. Hence $A_n$ is $\mathcal{C}$-tight since $A$ is continuous.
\end{proof}


\subsection{Expression for $\mu_v$}
Since $\disp \sum_{a\le x\le b} \{TG^2(x)-G^2(x)\}= 0$, it follows that
$\disp
\mu_v = \sum_{x\in\dZ}v(x) = \sum_{x\in\dZ}\left\{2V(x)G(x)-V^2(x)\right\}$.
One can check that
$\disp \sum_y\sum_z |z-y|V(y)V(z) = 4\sum_y y V(y) H(y)+2\sum_y yV^2(y)$,
where $\disp H(x)=\sum_{y<x}V(y)$. Note that $H\equiv 0$ outside $[a,b]$. Now, $G(x) = -2xH(x)+2\sum_{y<x }yV(y)-c$. As a result,
\begin{eqnarray*}
\mu_v &=&  -4\sum_x xV(x)H(x) +4 \sum_x \sum_{y<x}yV(y)V(x) - \sum_x V^2(x)\\
&=&  -8\sum_x xV(x)H(x) -4 \sum_x xV^2(x) - \sum_x V^2(x),\\
\end{eqnarray*}
which is the expression of Equation 10 in \cite{Dobrushin:1955}.
Furthermore, since
$\disp  0 = \sum_x V^2(x)+2 \sum_x\sum_{y<x}V(y)V(x) = \sum_x V^2(x)+2\sum_x V(x)H(x)$, one also gets
\begin{eqnarray*}
2\sum_x H^2(x) &= & 2\sum_{a\le x\le b} \sum_{a\le y<x} \sum_{a\le z<x}V(y)V(z) = 2\sum_{a\le y\le b} \sum_{a\le z \le b}V(y)V(z)\sum_{b\ge x>\max(y,z)}\\
& =&   2\sum_{a\le y\le b} \sum_{a\le z\le b}V(y)V(z)\{b-\max(y,z)\} \\
 &=&4 \sum_z\sum_{y<z} V(y)V(z)(b-z) + 2\sum_{y} V^2(y)(b-y) \\
 &=&  -4\sum_{z} zH(z)V(z) -  2\sum_{y} yV^2(y),
\end{eqnarray*}
proving the expressions $c_V^2 = 2\sum_x H^2(x)$ and $\disp 2c_V^2 = \mu_v+\sum_x V^2(x)$.
Note that $\mu_v = 2B_V$, where
$\disp B_V = \frac{1}{2}\sum_{x\in\dZ}\left\{2H(x)+V(x)\right\}^2$, as defined in \cite{Remillard:1990}.





\end{document}